\documentclass[11pt, leqno]{article}
\usepackage{amsmath,amssymb,amscd,amsthm, latexsym}

\newtheorem{theorem}{Theorem}[section]
\newtheorem{lemma}[theorem]{Lemma}
\newtheorem{prop}[theorem]{Proposition}

\newtheorem{conjecture}[theorem]{Conjecture}

\theoremstyle{definition}

\theoremstyle{remark}
\newtheorem{remark}[theorem]{Remark}

\numberwithin{equation}{section}

\begin{document}

\title{\bf Giroux correspondence, confoliations, and symplectic structures on $S^1\times M$}

\author{Jin Hong Kim}



\maketitle

\begin{abstract}
Let $M$ be a closed oriented 3-manifold such that $S^1\times M$ admits a symplectic structure $\omega$. The goal of this paper is to show that $M$ is a fiber bundle over $S^1$. The basic idea is to use the obvious $S^1$-action on $S^1\times M$ by rotating the first factor of $S^1\times M$, and one of the key steps is to show that the $S^1$-action on $S^1\times M$ is actually symplectic with respect to a symplectic form cohomologous to $\omega$. We achieve it by crucially using the recent result or its relative version of Giroux about one-to-one correspondence between open book decompositions of $M$ up to positive stabilization and co-oriented contact structures on $M$ up to contact isotopy. As a consequence, we can give an answer to a question of Kronheimer concerning the relation between symplectic structures on $S^1\times M_K$ and fibered knots $K$, where $M_K$ denotes the result of $0$-surgery on $S^3$ along a knot $K$ in $S^3$. Moreover, a complete picture of the various intriguing implications between symplectic structures on $S^1\times M_K$ and fibered knots can be provided as in Table I below, and thus we fill in the missing links in the circle of ideas around this topic.
\end{abstract}


\section{Introduction and statements of results} \label{sec1}

It was Thurston in \cite{Th} who first proved that any closed
oriented smooth 4-manifold $X$ which fibers over a Riemann surface
admits a symplectic structure, unless the fiber class is torsion in
$H_2(X, {\bf Z})$. Thus, if the genus of the fiber of such a closed
oriented 4-manifold $X$ is greater than or equal to 2, then the
manifold $X$ always admits a symplectic structure. Moreover, any
fibration of a closed oriented 3-manifold $M$ over a circle $S^1$
induces a symplectic structure on the 4-manifold $S^1\times M$. (See
\cite{Mc-S} and \cite{G-S}.)  Furthermore, it seems to have been widely believed
that the converse also holds (see \cite{Ta}).

\begin{conjecture} \label{conj1.1}
Let $M$ be a closed oriented $3$-manifold such that $S^1\times M$
admits a symplectic structure. Then $M$ fibers over $S^1$.
\end{conjecture}

Indeed, there have been several attempts towards the conjecture, and it
turns out that the conjecture is true for many important classes of
4-manifolds. For instance, see \cite{McC}, \cite{C-M}, \cite{Et},
\cite{F-V}, \cite{F-V2}, and \cite{F-V3}.

The goal of this paper is that by taking a completely different but elementary approach from the previous work, we give a short and affirmative proof to the Conjecture \ref{conj1.1}. To do so, we shall use the $S^1$-action on the symplectic 4-manifold $S^1\times M$ obtained by rotating the first factor of $S^1\times M$. One crucial observation of the proof is that every symplectic class on $S^1\times M$ can always be represented by a symplectic form invariant under the $S^1$-action. We show this important fact by a recent result and its relative version of Giroux about one-to-one correspondence between contact structures up to contact isotopy and open book decompositions up to positive stabilizations (see \cite{Gi}, \cite{HKM}, \cite{EO}, and \cite{Etn} for more details).

To the author's knowledge, it is still unknown whether or not any action of a compact connected Lie group $G$ on a symplectic $2n$-manifold $X$ always induces a $G$-invariant symplectic form on
$X$, in general. On the other hand, in Riemannian geometry one can always obtain
a $G$-invariant Riemannian metric by taking the average of a
Riemannian metric over the Lie group $G$. As $S^2\times S^2$ with
the product symplectic form shows, simply taking the average of a
symplectic form over the Lie group $G$ does not yield a
$G$-invariant symplectic form \cite{Ono}.

Once we show that there exists a symplectic form invariant under an $S^1$-action on $S^1\times M$, it is an
easy and well-known procedure to complete the proof of the conjecture. In Section \ref{sec4}, for the sake of the reader we provide a proof of Theorem \ref{thm4.1}, using an argument of D. Tischler in
\cite{Ti} about fibering certain foliated manifolds over $S^1$. We
remark that Theorem 18 in \cite{F-G-M} gives an alternative argument
of the second half of the proof of our main Theorem \ref{thm4.1}.

As a generalization of the Theorem \ref{thm4.1}, Baldridge asked in
\cite{Ba} whether or not for every closed symplectic 4-manifold
admitting a free $S^1$-action whose orbit space is $M$ the quotient
manifold $M$ fibers over $S^1$. We think that our method of the present paper can be adapted to answer the following Conjecture \ref{conj1.2}. But we do not pursue it in this paper, for the sake of simplicity.

\begin{conjecture} \label{conj1.2}
Let $X$ be a closed symplectic $4$-manifold admitting a free
$S^1$-action whose orbit space is $M$. Then the quotient manifold
$M$ fibers over $S^1$.
\end{conjecture}

Let us denote by $M_K$ a 3-manifold $M_K$ obtained by 0-surgery on a knot $K$ in $S^3$.
As an interesting consequence, we can give an answer to the question
in \cite{K} of Kronheimer about symplectic structures on $S^1\times
M_K$. To be more precise, Fintushel and Stern proved in
\cite{F-S} that if $S^1\times M_K$ admits a symplectic structure
then the symmetrized Alexander polynomial $\Delta_K(t)$ is monic. On
the other hand, Kronheimer proved in \cite{K} that if the knot has a
genus $g(K)$ of two or more a necessary condition for $S^1\times
M_K$ to admit a symplectic structure is that its genus $g(K)$ be
equal to the degree of its symmetrized Alexander polynomial. It is a
well-known fact (\cite{B-Z} or \cite{Ro}) that if a knot $K$ is
fibered then its symmetrized Alexander polynomial is monic and its
genus $g(K)$ is equal to the degree of is symmetrized Alexander
polynomial. Since $S^1\times M_K$ is symplectic for fibered knots,
Kronheimer raised a question whether or not $S^1\times M_K$ admits a
symplectic structure for \emph{non-fibered} knots such as the pretzel knot
$P(5, -3, 5)$. The symmetrized Alexander polynomial of the pretzel
knot $P(5, -3, 5)$ is $t-3+t^{-1}$ and thus monic with its degree
equal to the genus $1$ of the knot. Our another main result is to give a
negative answer to the question of Kronheimer as follows:

\begin{theorem} \label{thm1.1}
The product $4$-manifold $S^1\times M_K$ admits a symplectic structure
if and only if the knot $K$ is always fibered.
\end{theorem}

According to the recent paper \cite{F-V} of S. Friedl and S.
Vidussi, the product of $S^1$ with the $0$-surgery of $S^3$ along
the pretzel knot $P(5, -3, 5)$ does not admit a symplectic
structure, which fits well with our result. We give the proof of Theorem \ref{thm1.1} at the end of Section 4. As a corollary, as the pretzel knot $P(5, -3, 5)$ shows, the statement that the symmetrized Alexander polynomial of a knot $K$ is monic and its genus $g(K)$ is equal to the degree of its symmetrized Alexander polynomial does not
imply that $S^1\times M_K$ admits a symplectic structure. In summary, when we set the statements {\bf (A)}, {\bf (B)}, and {\bf (C)} as follows,
\begin{itemize}
 \item[\bf (A)] $S^1\times M_K$ admits a symplectic structure,
 \item[\bf (B)] The symmetrized Alexander polynomial of a knot $K$ with
 genus $\ge 2$ is monic and its knot genus $g(K)$ is equal to the degree of its
symmetrized Alexander polynomial,
 \item[\bf (C)] The knot $K$ is fibered,
 \end{itemize}
we can establish the following table for the various implications:
\begin{center}
{Table I}
\vskip .3cm
\begin{tabular}{|c|c|c|} \hline

Implication & True/False
& Reason \\
\hline\hline

{\bf (A)} $\rightarrow$ {\bf (B)} & True & Proved by
Kronheimer and Fintushel-Stern \\
\hline

{\bf (B)} $\rightarrow$ {\bf (A)} & False
& e.g., Pretzel knot $P(5,-3, 5)$ \\
\hline

 {\bf (B)} $\rightarrow$ {\bf (C)} & False
& e.g., Pretzel knot $P(5,-3, 5)$ \\
\hline

 {\bf (C)} $\rightarrow$ {\bf (B)}$\dagger$ & True &
 Proposition 8.16 in \cite{B-Z} (Neuwirth)\\
\hline

 {\bf (A)} $\rightarrow$ {\bf (C)} & True
& Theorem \ref{thm1.1} \\
\hline

 {\bf (C)} $\rightarrow$ {\bf (A)} & True
& Proved by Thurston \\
\hline
\end{tabular}
\end{center}
\qquad $\dagger${\small For this direction, we do not need the
restriction on the genus of a knot.}
\bigskip

Finally a few remarks are in order. During the preparation of our paper, two papers related to the Conjecture \ref{conj1.1} have appeared. In their paper \cite{Ku-Ta}, Kutluhan and Taubes studied the Seiberg-Witten Floer homology of $M$, assuming that $S^1\times M$ admits a symplectic form. As a consequence, by combining their results with Theorem 1 of Y. Ni in \cite{Ni}, they gave a different proof that $M$ fibers over $S^1$, in case that $M$ has the first Betti number equal to 1 and the first Chern class of the canonical line bundle is not torsion.
Friedl and Vidussi also posted a preprint \cite{Fr-Vi08} asserting the proof of Conjecture \ref{conj1.1} modulo some technical step regarding the residually finite solvability of $\pi_1(M)$ which allegedly depends on a work under preparation by M. Aschenbrenner and S. Friedl. Among other things, the twisted Alexander polynomials, algebraic group theory, and Stallings' characterization (\cite{St}) for the fibration of a $3$-manifold over a circle play crucial roles in their proof.
\smallskip

We organize this paper as follows. In Section \ref{sec2}, we give some basic facts about open
book decompositions for closed contact 3-manifolds, partial open book decompositions for compact contact $3$-manifolds with convex boundary, and confoliations. Section \ref{sec3} is one of the key sections for this paper. In that section, we show that every symplectic class on $S^1\times M$ can always be represented by a symplectic form invariant under the naturally defined $S^1$-action. Finally Section \ref{sec4} is devoted to the proofs of the main Theorems \ref{thm4.1} and \ref{thm4.2}.

\section{Giroux correspondence and confoliations} \label{sec2}

The aim of this section is to review some basic facts about open
book decompositions for contact 3-manifolds, partial open book decompositions for compact contact $3$-manifolds with convex boundary, and confoliations.

First we briefly review the definition of an open book decomposition
of a closed 3-manifold $M$, and its extension to compact contact 3-manifolds with convex boundary
can be easily obtained with an obvious modification (see the recent papers \cite{HKM} and \cite{EO}). Let $(F, h)$ be a pair consisting of an oriented surface $F$ and a diffeomorphism
$h:F\to F$ which is the identity on $\partial F$, and $K$ be a link
in $M$. An open book decomposition for $M$ with binding $K$ is the
quotient space
\[
((F\times [0,1])/\sim_h, (\partial F\times [0,1])/\sim_h))
\]
which is homeomorphic to $M$. Here the equivalence relation $\sim_h$
is given by
\begin{equation*}
\begin{split}
&(x,1)\sim_h (h(x), 0) \ \text{for}\ x\in F\ \text{and}\\
&(x, t)\sim_h(x, t')\ \text{for}\ x\in \partial F\ \text{and}\ \text{all}\ t, t'\in [0,1].
\end{split}
\end{equation*}
We will call $F\times \{ t \}$ for $t\in [0,1]$ a \emph{page} of the open
book decomposition. Two open book decomposition is \emph{equivalent} if
there is an ambient isotopy between them taking binding to binding
and pages to pages. We can obtain a new open book decomposition $(F,
h')$ from $(F, h)$ by a \emph{positive} (resp. \emph{negative}) \emph{stabilization}.
Namely, $F'$ is obtained from $F$ by attaching a 1-handle $B$ along
$\partial F$ and $h'$ is obtained by extending $h$ by the identity
map on the 1-handle $B$ and taking the composition $R_\gamma\circ h$
(resp. $R_\gamma^{-1}\circ h$) with the right-handed Dehn twist
$R_\gamma$ along a simple closed curve $\gamma$ in $F'$ dual to the
core of the 1-handle $B$.

It is known that every closed 3-manifold has an open book
decomposition, but it is not unique. A contact structure $\tau$ is
said to be \emph{supported} (or \emph{adapted}) by the open book
decomposition $(F, h, K)$ if there is a contact 1-form $\lambda$
satisfying the following properties:
\begin{itemize}
\item $\lambda$ induces a symplectic form $d\lambda$ on each fiber
$F$.

\item $K$ is transverse to $\tau$, and the orientation on $K$ given
by $\lambda$ is the same as the boundary orientation induced from
$F$ coming from the symplectic structure.
\end{itemize}

Thurston and Winkelnkemper showed in \cite{TW75} that any open book
decomposition $(F, h, K)$ supports a contact structure. The contact
planes constructed by them can be made arbitrary close to the
tangent planes of the pages away from the binding. Recently E.
Giroux showed in \cite{Gi} that the converse also holds. To be more
precise, the following theorem holds:

\begin{theorem} \label{thm2.1}
Every contact structures $\tau$ on a closed $3$-manifold $M$ is
supported by some open book decomposition $(F,h,K)$. Moreover, two
open book decompositions $(F, h,K)$ and $(F', h', K')$ which support
the same contact structure $(M,\tau)$ become equivalent after
applying a sequence of positive stabilizations to each.
\end{theorem}

In our situation, we do not use the full version of this theorem.
Rather we will need the result of Giroux to choose a coordinate
chart on $S^1\times M$ with which we can easily calculate the Lie
derivative of the symplectic form for our purposes (e.g., see Lemma
\ref{lem3.3} for more details). Even if the above Theorem
\ref{thm2.1} is stated for closed contact 3-manifold, the
construction of Giroux shows that the same result holds for
contact 3-manifolds with contact boundary, as the papers \cite{HKM} and \cite{EO} of Honda-Kazez-Mati\' c and Etg\" u-Ozbagci show.

For the sake of reader's convenience, we briefly review the relative Giroux correspondence for compact contact $3$-manifolds with convex boundary, although in the present paper we do not need the full strength of this correspondence. For more details, see \cite{HKM} and \cite{EO}, and most of what is presented here can be found in those two papers.

We first begin with the abstract version of a \emph{partial open book decomposition} which is a triple $(S, P, h)$ satisfying the following three properties:
\begin{itemize}
  \item $S$ is a compact oriented connected surface with non-empty boundary $\partial S$,
  \item $P=P_1\cup P_2\cup\cdots \cup P_r$ where $P_1$, $P_2$, $\ldots$, $P_r$ are $1$-handles is a proper, but not necessarily connected, subsurface of $S$ such that $S$ is obtained from the closure of $S\backslash P$ by attaching 1-handles $P_1$, $P_2$, $\ldots$, $P_r$ successively,
  \item $h: P\to S$ is an embedding such that $h|_{\partial P\cup\partial S}=$identity.
\end{itemize}
Given a partial open book decomposition $(S, P, h)$, we can construct a compact $3$-manifold with boundary as follows. Let $H=(S\times [-1,0])/\sim$, where $(x, t)\sim (x, t')$ for all $x\in \partial S$ and $t, t'\in [-1,0]$,  which is a solid handlebody with $S\times \{ 0 \}\cup -S \times \{-1\}$ as the boundary under the obvious relation $(x, 0)\sim (x, -1)$ for all $x\in \partial S$. We also let $N=(P\times [0,1])/\sim$, where $(x,t)\sim (x,t')$ for all $x\in \partial P\cap \partial S$ and $t,t'\in [0,1]$. Again each component of $N$ is a solid handlebody whose boundary can be described by the connected arcs of the closure of $\partial P\backslash \partial S$. In other words, let $c_1, c_2, \cdots, c_n$ denote such connected arcs. Then each disk $D_i= (c_i\times [0,1])/\sim$ is contained in the boundary of $N$. Thus the boundary of $N$ consists of the union of the disjoint disks $D_i$'s and the surface $P\times \{ 1\}\cup -P\times \{0 \}$  with the relation $(x, 0)\sim (x, 1)$ for all $x\in \partial P\cap \partial S$.

Now let $M=N\cup H$ with the identification of $P\times \{ 0\}\subset \partial N$ (resp. $P\times \{ 1\}\subset \partial N$) with $P\times \{ 0\}\subset \partial H$ (resp. $h(P)\times \{ -1\}\subset \partial H$). Then $M$ is an oriented compact $3$-manifold with oriented boundary
\begin{equation} \label{eq2.1}
\partial M= (S\backslash P)\times \{ 0 \} \cup -(S\backslash h(P))\times \{ -1 \} \cup (\overline{\partial P\backslash \partial S})\times [0,1]
\end{equation}
with the suitable identifications. If a compact $3$-manifold $M$ with boundary is obtained from the abstract partial open book decomposition $(S, P, h)$ as above, then the triple $(S, P, h)$ is called a \emph{partial open book decomposition} of $M$. The notions such as compatibility of a contact structure with respect to a partial open book decomposition, the isomorphism class of two partial open book decompositions, and the definition of a positive stabilization of a partial open book decomposition can also be interpreted suitably for this relative version (e.g., see Definitions 1.10, 1.11, and1.13 in \cite{EO}).

Recall that a closed oriented embedded surface $\Sigma$ in a contact manifold $(M,\xi)$ is called \emph{convex} if there is vector field tansverse to $\Sigma$ which preserves the contact structure $\xi$. A generic surface $\Sigma$ inside a contact manifold can be made convex (cf. \cite{Gi91} and Section 2.2 of \cite{Honda99}). So the assumption that the boundary be convex can be imposed without loss of generality.

In \cite{HKM}, Honda-Kazez-Mati\' c associated the isomorphism classes of compact contact $3$-manifolds with convex boundary to the isomorphism classes of partial open book decompositions modulo positive stabilizations. Conversely, in \cite{EO} Etg\" u and Ozbagci constructed its inverse by describing a compact contact $3$-manifold with convex boundary compatible with a given partial open book decomposition. As in the proof of Proposition 1.9 in \cite{EO}, such a construction is essentially given by the explicit construction of Thurston and Winkelnkemper. Hence the property, as well as others, that for closed contact $3$-manifolds the contact planes constructed by them can be made arbitrary close to the tangent planes of the pages away from the binding can also be used for compact $3$-manifolds with convex boundary.

Now we can state a relative version of Giroux correspondence as follows, which is a relative version of Giroux correspondence for closed contact $3$-manifolds (see Theorem 0.1 in \cite{EO}).

\begin{theorem} \label{thm2.2}
There is a one-to-one correspondence between isomorphism classes of partial open book decompositions modulo positive stabilization and the isomorphism classes of compact contact $3$-manifolds with convex boundary.
\end{theorem}

In what follows, we will also need to use the result of Eliashberg
and Thurston (Theorem 2.4.1 in \cite{E-T97}). A plane field
$\eta=\ker \theta$ on an oriented 3-manifold is called
\emph{positive} (resp. \emph{negative}) \emph{confoliation} if
$\theta\wedge d\theta\ge 0$ (resp. $\theta\wedge d\theta\le 0$). Let
us denote by $\zeta$ the product foliation of the manifold
$S^2\times S^1$ by the spheres $S^2\times \{ z \}$ for $z\in S^1$.

\begin{theorem}[Eliashberg-Thurston] \label{thm2.3}
Suppose that a $C^2$-confoliation $\eta$ on an oriented 3-manifold
is different from the foliation $\zeta$ on $S^2\times S^1$. Then
$\eta$ can be $C^0$-approximated by contact structure. When $\eta$
is a foliation it can be $C^0$-approximated both by positive and
negative contact structure.
\end{theorem}

\section{Existence of $S^1$-invariant symplectic structures} \label{sec3}

Recall that there exists an obvious circle action on the 4-manifold $S^1\times M$ obtained by rotating the first factor of $S^1\times M$. The aim of this section is to show that every symplectic class on $S^1\times M$ can be represented by a symplectic form which is invariant under the obvious action of $S^1$.

In what follows, we assume that $S^1\times M$ admits a symplectic
structure $\omega$. If $M$ is $S^2\times S^1$, then clearly $M$ fibers over $S^1$. So from now on we also assume that $M$ is not $S^1\times S^2$, unless stated otherwise.  Let $X$ be the fundamental vector field associated to
the action of $S^1$, and let $\alpha=\iota_X \omega$. Since $\omega$ is a symplectic 2-form, $\alpha$ is clearly a nowhere vanishing 1-form on $S^1\times M$. Now choose an arbitrary point $t$ in $S^1$. Let $j_t$ denote the inclusion from $M$ into $S^1\times M$ given by $x\mapsto (t, x)$. Then we obtain
a nowhere vanishing 1-form $\beta_t$ by the pull-back of the 1-form
$\alpha$ restricted to $\{ t \}\times M$ via the inclusion $j_t$. In
other words, $\beta_t=j_t^\ast (\alpha|_{\{ t \}\times M})$. Then we have the following proposition.

\begin{prop} \label{prop3.1}
The differential 3-form $\beta_{t}\wedge d_M\beta_{t}$ should vanish identically for all $t\in S^1$.
\end{prop}

\begin{proof}
We divide the proof into the following three cases:
\smallskip

\noindent{\bf (Case 1)} First of all, assume that $\beta_t\wedge d_M
\beta_t$ is non-zero for all $t$. Thus the 2-plane field $\xi_t=\ker
\beta_t$ is a family of contact structures on $M$. It is obvious
that $d_M \beta_t$ is a nowhere vanishing 2-form on $M$. Moreover,
the following holds.

\begin{lemma} \label{lem3.1}
Under our assumption, the Lie derivative ${\mathcal L}_X \omega$ is
nowhere vanishing on $S^1 \times M$.
\end{lemma}

\begin{proof}
The proof follows from the Cartan's formula. Indeed, it suffices to note that
\begin{equation} \label{eq3.1}
\begin{split}
0 &\ne d_M\beta_t =d_M j_t^\ast (\iota_X \omega)=j_t^\ast d \iota_X
\omega=j_t^\ast (d\iota_X \omega+ \iota_X d\omega)\\
&=j_t^\ast ({\mathcal L}_X \omega).
\end{split}
\end{equation}
This completes the proof of Lemma \ref{lem3.1}.
\end{proof}

Next we can show the following

\begin{lemma} \label{lem3.2}
The symplectic $2$-form $\omega$ restricted to the contact structure
$\xi_t=\ker \beta_t$ along $\{ t \} \times M$ is non-zero.
\end{lemma}

\begin{proof}
To see it, first note that there exists a Reeb vector field $Z_t$ on
$M$, depending on the parameter $t$, such that
\begin{equation} \label{eq3.2}
1=\beta_t(Z_t)=j_t^\ast (\iota_X \omega)(Z_t)=\omega(X,
(j_t)_\ast(Z_t)).
\end{equation}
Since, along each point $(t, x)$ in $S^1\times M$, the vector space
spanned by $X$ and $(j_t)_\ast(Z_t)$ is transversal to the contact
plane $\xi_t$ and the equation \eqref{eq3.2} is satisfied,  we conclude that the restriction $\omega|_{\xi_t}$ of the symplectic form $\omega$ on $S^1\times M$ is non-zero.
\end{proof}

Now, we apply the result of Giroux in \cite{Gi} concerning the open
book decomposition of a contact 3-manifold (Theorem \ref{thm2.2}).
In our situation, we can choose a family of open book decompositions
along a connected binding $B_t$ associated to the contact structure
$\beta_t$ on $M$, so that the parameter $s_t$ for the base manifold
$S^1$ for the fibration associated to the open book decomposition is
given by the Reeb vector field $Z_t$. We recall that by the way of
the construction of the open book decomposition the contact plane
$\xi_t$ can be made arbitrarily close to the pages $F_t$. Thus it
follows from Lemma \ref{lem3.2} that, along $\{ t \} \times M$,
$\omega$ restricted to the pages $F_t$ of the open book
decomposition is also non-zero. That is, we see that, along $\{ t \}
\times M$, $\omega$ restricted to the pages $F_t$ of the open book
decomposition is a volume form away from the binding $B_t$. Recall also that $d_M\beta_t$ is a
volume form on the pages $F_t$ away from the binding $B_t$ by the construction of the open book
decomposition.

Let $M_{B_t}$ be the result of 0-surgery along $B_t$. Then we have a
fibration
\[
\pi: S^1\times M_{B_t} \to S^1\times S^1,
\]
and $t$ and $s_t$ will denote the first and second angular
coordinates on the base manifold $S^1\times S^1$ of the fibration
$\pi$, respectively. Let $N(B_t)$ denote a tubular neighborhood of
the binding $B_t$. Since, along $\{ t \} \times M$, $d_M\beta_t$ and
$\omega$ are both nowhere vanishing 2-forms restricted to the
contact plane $\xi_t$, we can choose a smooth function $f$ defined
over $S^1\times (M_B\backslash N(B))$ satisfying the following two
properties:

\begin{itemize}
\item $f$ is nowhere vanishing over $S^1\times (M_{B_t}\backslash N(B_t))$
and
\item $f\cdot d_M\beta_t$ coincides with $\omega$, when restricted to the
contact plane $\xi_t$.
\end{itemize}

Then the following lemma holds.

\begin{lemma} \label{lem3.3}
On the manifold $S^1\times (M_{B_t}\backslash N(B_t))$ which can be
identified with $S^1\times (M\backslash N(B_t))$, the symplectic
$2$-form $\omega$ can be written locally as the form
\begin{equation} \label{eq3.3}
\omega=\pi^\ast(dt\wedge ds_t)+ f(t,x) d_M\beta_t + ds_t\wedge
\delta,
\end{equation}
where $\delta$ is a 1-form on $S^1\times (M\backslash N(B_t))$ and
$x$ denotes a local coordinate on $M$.
\end{lemma}

\begin{proof}
To see it, notice first that $\omega(Z_t, W_t)$ may be non-zero for
the Reeb vector field $Z_t$ and $W_t\in\xi_t$, while $\omega(X,
W_t)$ should be zero for $W_t\in\xi_t$. Thus in local coordinates
the symplectic form $\omega$ should have only the terms involving
$\pi^\ast(dt\wedge ds_t)$, $d_M\beta_t$ and $ds_t\wedge \delta$. Due
to the equation \eqref{eq3.2}, the coefficient of $\pi^\ast(dt\wedge
ds_t)$ should be 1, as stated. Thus we are done.
\end{proof}

Finally, over $S^1\times (M_{B_t}\backslash N(B_t))$ we compute the
Lie derivative ${\mathcal L}_X \omega$ explicitly. To do so, note
that we have
\begin{equation} \label{eq3.4}
\begin{split}
{\mathcal L}_X \omega &= d \iota_X(\pi^\ast(dt\wedge
ds_t)+f(t,x) d_M\beta_t+ ds_t\wedge \delta)\\
&=d (-\iota_X\delta)\wedge ds_t.
\end{split}
\end{equation}
However, since we have
\[
1=\omega(X,(j_t)_\ast(Z_t))=1-\iota_X\delta
\]
by the equations \eqref{eq3.2} and \eqref{eq3.3}, $\iota_X\delta$
should be zero. Thus it follows from \eqref{eq3.4} that ${\mathcal
L}_X \omega=0$. This clearly contradicts to Lemma \ref{lem3.1}.
That is, this case does not occur.

\smallskip

\noindent {\bf (Case 2)} We next assume that, for some $t=t_0$ in $S^1$, $\beta_{t}\wedge d_M\beta_{t}$ is non-zero for all $x\in M$. By the continuity of smooth differential forms, $\beta_{t}\wedge d_M\beta_{t}$ should be non-zero for all $t$ in some sufficiently small open interval $I$
of $t_0$ and all $x\in M$.

Then apply the arguments in (Case 1) to the manifold $I\times M$ instead of $S^1\times M$. Then we can also derive a contradiction in this case. In more detail, the manifold $I \times M$ admits a symplectic structure, denoted $\omega$, by the restriction of the symplectic form $\omega$ on $S^1\times M$ to $I\times M$. There still exists the fundamental vector field $X$ on $I\times M$ associated to the natural action of $S^1$ on $S^1\times M$, since the interval is regarded as an open submanifold of $S^1$. However, clearly there is no $S^1$-action on $I$. Using this fundamental vector field $X$ on $I\times M$ and the fact that $\beta_{t}\wedge d_M\beta_{t}$ is non-zero on $I\times M$, as in Lemma \ref{lem3.1} we can show that the Lie derivative ${\mathcal L}_X \omega$ is nowhere vanishing on $I\times M$. Furthermore, one can check that other arguments as well as Lemmas \ref{lem3.2} and \ref{lem3.3} go through without any modification. So, we can conclude that this case does not occur, either.

\smallskip

\noindent {\bf (Case 3)} In this case we assume that, for some $t=t_0$ in
$S^1$ and some $x_0\in M$, $\beta_{t}\wedge d_M\beta_{t}$ is non-zero. Once again it follows from the continuity of smooth differential forms that $\beta_{t}\wedge d_M\beta_{t}$ should be non-zero for all
$t$ in some sufficiently small open interval $I$ of $t_0$ and some $x_0\in M$.

In order to apply the arguments of the previous cases, we need to take the
contact part
\[
V(\beta_t)=\{ x\in M \ | \ \beta_{t}\wedge d_M\beta_{t}\ne 0\
\text{for}\ t\in I \}.
\]
For simplicity,  for each $t\in I$ we shall denote by $W_t$ the closure of the connected component of the contact part $V(\beta_t)$ which contains $x_0$. Then for each $t\in I$, $W_t$ is a compact contact submanifold of $M$ of codimension 0 with (possibly empty) boundary, since $\beta_{t}\wedge d_M\beta_{t}$ is a nowhere vanishing $3$-form  on $V(\beta_t)$ and so is a volume form there. Fortunately, in the present paper we do not need to know the precise information of $W_t$, unlike to the case in the paper \cite{H-T}.

Assume now that the boundary is non-empty. (Otherwise, we are reduced to (Case 2).) Then, as already mentioned in Section \ref{sec2}, we may assume without loss of generality that the boundary is convex. Thus for each $t\in I$ we obtain a compact contact $3$-manifold $W_t$ with convex boundary. It is also true that as in (Case 2) above there still exists the fundamental vector field $X$ on $\cup_{t\in I} \{ t \}\times W_t\subset S^1\times M$ associated to the natural action of $S^1$ on $S^1\times M$, since the interval $I$ is again regarded as an open submanifold of $S^1$. But this case is slightly different from (Case 2) in that $W_t$ is a compact contact $3$-manifold with convex boundary. So we need to use the relative Giroux correspondence (Theorem \ref{thm2.2}) instead of the Giroux correspondence (Theorem \ref{thm2.1}). In other words, for each $t\in I$ apply Theorem \ref{thm2.2} to $W_t$ in order to obtain its partial open book decomposition $(S_t, P_t, h_t)$. Thus $W_t$ can now be described as the gluing of two handlebodies $H_t$ and $N_t$ by the map $h_t$ whose boundary is given as in \eqref{eq2.1}. However, since all the arguments in (Case 1) and (Case 2) are essentially local, those arguments applied to the compact contact $3$-manifold $W_t$ with convex boundary and symplectic $4$-manifold $\cup_{t\in I} \{ t \} \times W_t$ equipped with the symplectic form induced from $S^1\times M$ will again go through without any modification. This in turn gives rise to a contradiction for this case, which means that this case does not occur, either.

This completes the proof of Proposition \ref{prop3.1}.
\end{proof}

With this understood, the following theorem will play a crucial role in the proof of Theorem \ref{thm4.1}.

\begin{theorem} \label{thm3.1}
The symplectic class $[\omega]$ on $S^1\times M$ can be represented by a symplectic form which is invariant under the obvious action of $S^1$.
\end{theorem}

\begin{proof}
We prove this theorem by contradiction. That is, suppose that the cohomology class $[\omega]$ cannot be represented by any $S^1$-invariant symplectic form $\omega$ under the obvious $S^1$-action. Then we would have
\begin{equation} \label{eq3.5}
{\mathcal L}_X\omega\ne 0.
\end{equation}
Note also that by Proposition \ref{prop3.1} the differential 3-form $\beta_{t}\wedge d_M\beta_{t}$ vanishes identically for all $t\in S^1$. Then there are two possibilities we have to consider: $d_M\beta_t$ vanishes identically on $\{ t \}\times M$ for all $t\in S^1$ or not.

So, suppose first that $d_M\beta_t$ vanishes identically on $M$ as well. Then it can be shown that the Lie derivative ${\mathcal L}_X \omega$ vanishes identically. To see it, notice that it follows from the identity \eqref{eq3.1} that we have $j_t^\ast({\mathcal L}_X \omega)=0$. Thus we have ${\mathcal L}_X\omega(Z_t, W_t)=0$ for any vector fields $Z_t$ and $W_t$ on $\{ t \}\times M$ for each $t\in S^1$.
Moreover, since $\omega$ is a symplectic form on $S^1\times M$, we can choose a Darboux chart in a neighborhood of a point $(t, x)$ whose coordinate vectors are given by non-zero vector fields $X_0=X$ and $X_i$ $(i=1,2,3)$. With this coordinate chart, we have
\[
{\mathcal L}_X \omega(X, X_i)=d\iota_X\omega(X, X_i)=X(\iota_X \omega(X_i))=X(1)=0,
\]
which implies that ${\mathcal L}_X\omega(X, Y_t)=0$ for all vector field $Y_t$ of $\{ t \}\times M$. Therefore, we can conclude that in this case ${\mathcal L}_X\omega$ vanishes identically. But this clearly contradicts to the assumption \eqref{eq3.5}.

On the other hand, if $d_M\beta_t$ does not vanish on $\{ t \}\times M$, we need to use the result of Eliashberg and Thurston about perturbing a confoliation into a contact structure. If the 3-manifold
$M$ is $S^2\times S^1$, clearly $M$ fibers over $S^1$, as mentioned earlier. Thus we may
assume that our foliation $\xi_t$ is different from the foliation
$\zeta$ on $S^2\times S^1$. Now if we apply Theorem \ref{thm2.3} to
$\xi_t$ then we have a contact structure $\tilde\xi_t=\ker \tilde
\beta_t$ which is a $C^0$-approximation to $\xi_t$. Since
$\tilde\xi_t$ is a $C^0$-approximation of $\xi_t$, the symplectic
2-form $\omega$ can also be $C^0$-approximated by a symplectic
2-form $\tilde\omega$ on $S^1\times M$ so that
$\tilde\beta_t=j^\ast_t(\iota_X \tilde\omega|_{ \{ t \} \times M})$.
So we are essentially led to the (Case 1), (Case 2), or (Case 3) of Proposition \ref{prop3.1}, which has already shown not to occur.

Therefore, for either case we have derived a contradiction under our assumption \eqref{eq3.5}. This completes the proof of Theorem \ref{thm3.1}.
\end{proof}

\begin{remark}
Note that Theorem \ref{thm3.1} does not imply that any arbitrary symplectic form $\omega$ on $S^1\times M$ is always $S^1$-invariant under the $S^1$-action on the first factor of $S^1\times M$. This can be easily seen by taking $M$ to be the $3$-dimensional torus $T^3$. That is, if the theorem implies that any symplectic form $\omega$ on $S^1\times T^3$ is always $S^1$-invariant under the obvious $S^1$-action, the symplectic form on $S^1\times T^3=T^4$ should also be invariant under the obvious $S^1$-action of the last three $S^1$-factors of $T^4$. So we can conclude that every symplectic form on $T^4$ should be invariant under the componentwise $T^4$-action on $T^4$. But obviously this is not the case for $T^4$.
\end{remark}

\section{Proofs of Theorems \ref{thm4.1} and \ref{thm4.2}}  \label{sec4}

In this section we present the proofs of main theorems of the
present paper. Once we have established the existence of an $S^1$-invariant symplectic structure on $S^1\times M$, it is a fairly standard procedure to complete the proof of Conjecture \ref{conj1.1}. For the sake of reader's convenience, we give its proof relatively in detail.

To do so, we begin with the following well-known lemma which says that when
the cohomology class $[\iota_X\omega]$ is not integral and non-zero, by some suitable perturbation we can always make it integral.

\begin{lemma} \label{lem4.1}
Let $\omega$ be an $S^1$-invariant symplectic form on a closed oriented $4$-manifold
$N$ such that $[\iota_X \omega]$ is non-zero. Then $N$ admits an $S^1$-invariant symplectic form $\hat \omega$ such that $[\iota_X \hat\omega]$ is non-zero and integral.
\end{lemma}

\begin{proof}
If $N$ admits an $S^1$-invariant symplectic form $\omega'$ such that
$[\omega']$ is rational, then we can easily obtain an
$S^1$-invariant symplectic form $\hat\omega$ such that
$[\hat\omega]$ is integral by multiplying some suitable integer to
$\omega'$. Note also that the class $[\iota_X \omega']$ is rational
if the class $[\omega']$ is.

So assume now that the class $[\omega]$ is not rational. It is clear that
there exists an arbitrary small closed 2-form $\eta$ such that
$\omega+\eta$ represent a rational cohomology class. Let $\hat \eta$
be the average of $\eta$ over the $S^1$-action. Since $S^1$ is
connected, for $\nu \in S^1$ $\nu^\ast\eta$ is a closed 2-form
representing the same cohomology class as $\eta$. Thus $\omega+
\hat\eta$ and $\omega+\eta$ have the same rational cohomology class.
Note also that $\omega'=\omega+\hat\eta$ is symplectic, provided
that $\eta$ is sufficiently small. By the openness of symplectic
condition again, we can further choose $\omega'$ in such a way that the class $[\iota_X
\omega']$ is non-zero. This completes the proof.
\end{proof}

Finally we are ready to prove the main theorems.

\begin{theorem} \label{thm4.1}
Let $M$ be a closed oriented $3$-manifold such that $S^1\times M$
admits a symplectic structure $\omega$. Then $M$ fibers over $S^1$.
\end{theorem}

\begin{proof}
By Theorem \ref{thm3.1}, we may assume that the symplectic
structure $\omega$ is $S^1$-invariant. Further, we may also assume that the
class $[\iota_X \omega]$ on $S^1\times M$ is integral by Lemma
\ref{lem4.1}.

In order to prove the theorem, we first consider the case where the class $[\iota_X \omega]$ is
zero. In this case, there exists a function, called the moment map $\mu:
S^1\times M\to {\bf R}$ such that $\iota_X \omega=d\mu$. Thus the
$S^1$-action is Hamiltonian. But it is clear that the $S^1$-action
on $S^1\times M$ does not have any fixed points that are critical points of $\mu$. This gives rise to a
contradiction to the fact that any Hamiltonian function on a closed symplectic
manifold should have at least two critical points (e.g., extremal
points). Therefore we can conclude that the class $[\iota_X \omega]$ is actually non-zero. Under this
condition, McDuff proved in \cite{Mc} that by using an argument of D.
Tischler in \cite{Ti}, there exists a generalized moment map $\mu:
S^1\times M\to S^1$ satisfying $\iota_X \omega=\mu^\ast(dt)$. Thus
by restricting the map $\mu$ to $\{\text{a point}\}\times M$, we
easily obtain a fibration of $M$ over $S^1$. This completes the
proof of Theorem \ref{thm4.1}.
\end{proof}

Now we close this section with a proof of Theorem \ref{thm1.1} as follows.

\begin{theorem} \label{thm4.2}
If $S^1\times M_K$ admits a symplectic structure, then $K$ is always
a fibered knot.
\end{theorem}

\begin{proof}
Suppose that $S^1\times M_K$ admits a symplectic structure. Then it
follows from Theorem \ref{thm4.1} that the 3-manifold $M_K$ is a
fibration of $S^1$. Moreover, by the construction in the proof of
Theorem \ref{thm4.1}, we have a closed 1-form $\iota_X\omega$ whose
class is integral and is pointwise non-zero. Now let $j:
S^3-N(K)\to M_K$ be the natural inclusion, where $N(K)$ is a tubular
neighborhood of $K$. Now observe that by the pullback we have a closed 1-form
$j^\ast(\iota_X\omega)$ on $S^3-N(K)$ whose class is still integral
and is non-zero pointwise on $S^3-N(K)$. Then $j^\ast(\iota_X\omega)$ defines a measured foliation ${\mathcal F}$ of $S^3-N(K)$ transverse to the boundary $\partial (S^3-N(K))$ with $T{\mathcal F}=\ker j^\ast(\iota_X\omega)$. Since $j^\ast(\iota_X\omega)$ is integral, we can also write $j^\ast(\iota_X\omega)=d\pi$ for a fibration $\pi: S^3-N(K)\to S^1$ whose fibers are the leaves of ${\mathcal F}$ (e.g., see p.2 of \cite{Mc-Ta} for more details). Thus the knot $K$ is indeed fibered. This completes the proof.
\end{proof}



\end{document}